\documentclass{amsart}
\usepackage{amsmath}
\usepackage{amssymb}
\usepackage{amsthm}
\usepackage[mathscr]{eucal}
\newtheorem{theorem}{Theorem}[section]
\newtheorem{lemma}[theorem]{Lemma}
\newtheorem{proposition}[theorem]{Proposition}
\newtheorem{corollary}[theorem]{Corollary}
\newtheorem*{remark}{Remark}

\newcommand{\K}{\mathcal K}
\newcommand{\V}{\mathcal V}
\newcommand{\W}{\mathcal W}
\newcommand{\R}{\mathbb R}
\newcommand{\X}{\mathfrak X}
\newcommand{\at}{\mathfrak{at}}
\newcommand{\qf}{\mathfrak{qf}}
\newcommand{\QFL}{{\rm QF}_L}
\newcommand{\QFLb}{\mathscr{QF}_L}
\newcommand{\Rep}{R_\varepsilon}
\newcommand{\SL}{{\mathcal S}_L}
\newcommand{\ML}{{\mathcal M}_L}
\newcommand{\St}{{\rm St}(\mathscr{QF}_L)}
\newcommand{\ModS}{{\rm Mod}_{\mathcal{S}_L}}
\newcommand{\ModM}{{\rm Mod}_{\mathcal{M}_L}}
\newcommand{\TL}{\mathcal{T}_L}
\newcommand{\GC}{\mathcal{G}_C}
\newcommand{\Gn}{\mathcal{G}_n}
\newcommand{\Cnu}{\mathfrak{C}_\nu}
\newcommand{\Ynu}{\mathfrak{Y}_\nu}
\newcommand{\Xnu}{\mathfrak{X}_\nu}
\newcommand{\Mnu}{\mathcal{M}_\nu}
\newcommand{\GH}{\mathcal{GH}}

\DeclareMathOperator{\core}{core}

\begin{document}

\title{The space of minimal structures}
\author{Oleg~Belegradek}
\thanks{}

\address{Department of Mathematics,
Istanbul Bilgi University, Dolapdere 34440,
Istanbul, Turkey}
\email{olegb@bilgi.edu.tr}

\date{}
\keywords{Minimal structure. Compactness theorem. Metric. Stone space. Compact space. Ultraproduct.
Gromov--Hausdorff topology.}
\subjclass{03C98, 51F99, 06E15, 03C20}

\begin{abstract}
For a signature $L$ with at least one constant symbol, an $L$-structure is called minimal
if it has no proper substructures. Let $\SL$ be the set of isomorphism types
of minimal $L$-structures. The elements of $\SL$
can be identified with ultrafilters of the
Boolean algebra of quantifier-free $L$-sentences, and therefore
one can define a Stone topology on  $\SL$.
This topology on $\SL$ generalizes the topology of the space of $n$\nobreakdash-marked groups.
We introduce a natural ultrametric on $\SL$, and show that
the Stone topology on $\SL$ coincide with
the topology of the ultrametric space $\SL$ iff  the ultrametric space $\SL$ is compact iff
$L$ is locally finite
(that is, $L$ contains finitely many $n$-ary symbols for any $n<\omega$).
As one of the applications of compactness of the Stone topology on $\SL$,
we  prove compactness of certain
classes of metric spaces
in the Gromov--Hausdorff topology.
This slightly refines the known result based on Gromov's ideas
that any uniformly totally bounded class of compact metric spases is precompact.
\end{abstract}

\maketitle
\section*{Introduction}
In the final remarks in his famous paper \cite{Gro1},
M.~Gromov explained how to deduce from the main result --- virtual nilpotency
of any finitely generated group of polynomial growth ---
the following more precise
version of the result:

\begin{quotation}
\emph{For any positive integers $k,d, n$,
there exists a positive integer $m$ such that
any $n$-generated group, in which for all $r=1,\dots,m$
the size of the ball of radius $r$ centered at the identity is at most $kr^d$,
has a subgroup of index and nilpotency class at most $m$.}
\end{quotation}

For a proof of that version,
he introduced and used a notion of limit of a sequence of groups
with distinguished $n$ generators. Implicitely, he defined a topology on
the class of such groups, and used its compactness, as well as
closedness of a certain subclass.
L.~van den Dries and A.~J.~Wilkie~\cite{DW} gave a new proof of the result above
by means of model-theoretic compactness theorem instead of
Gromov's topological compactness argument.

Formalizing Gromov's idea, R.~Grigorchuk~\cite{Gri} suggested a precise definition
of the topology used by M.~Gromov,
and showed that the defined topological space is
metrizable, separable, compact, and has a base consisting of clopen sets.
That topological space, the \emph{space of $n$-marked groups},
has been the subject of papers \cite{Ch, ChG, CGP}.

In the present paper we look at the space of marked groups from
a model-theoretic point of view, and introduce a more general \emph{space of minimal structures}.

For a signature $L$ containing at least one constant symbol,
an $L$-structure is called \emph{minimal} if it has no proper substructures.
For example, any $n$-marked group is a minimal $L$-structure, where $L$
is the language of groups with added $n$ constant symbols.

It is easy to show that the isomorphism type of a minimal $L$-structure is completely determined by its
quantifier-free theory (Proposition~\ref{qf}), and a set $S$ of quantifier-free $L$-sentences is the quantifier-free
theory of a minimal $L$-structure iff $S$ is a maximal finitely satisfiable
set of quantifier-free $L$-sentences (Proposition~\ref{max}).
The set $\SL$ of all such  $S$ can be equipped with a topology $\tau$,
a basis of which consists of the sets $\{S\in~\SL:~\phi\in S\},$
where $\phi$ is a quantifier-free $L$-sentence.
The topological space $(\SL,\tau)$ is naturally homeomorphic to
the Stone space of the Boolean algebra of quantifier-free $L$-sentences;
therefore it is compact and totally disconnected.
Therefore we call $(\SL,\tau)$ the Stone space of isomorphism types of minimal $L$-structures.
The space of isomorphism types of $n$-marked groups is just a clopen set in
the Stone space $\SL$ for a certain $L$.

We show that the `bounded' version of Gromov's theorem formulated above
can be deduced from its standard version using not model-theoretic compactness theorem
as it was done in~\cite{DW}, but only compactness of the Stone space $S_L$.

For any universally axiomatizable class $\K$ of $L$-structures, the set $\K_\star$ of
isomorphism types of minimal $L$-structures in $\K$ is closed in $\SL$
(Proposition~\ref{univ-theory}).
Let $\W$ be a variety of $L$-structures and $\V$ its subvariety.
We show
that $\V_\star$ is clopen in $\W_\star$ iff
the $\V$-free minimal $L$-structure
is finitely presentable in $\W$ (Proposition~\ref{fp}).
For example, for any group variety $\V$, the set of isomorphism types of $n$-marked $\V$-groups
is clopen in the space of isomorphism types of $n$-marked groups
iff the $\V$-free group of rank $n$ is finitely presentable.

For an arbitrary set $X$ of minimal $L$-structures, we characterize
in terms of ultraproducts the limit points of $X$ in the Stone topology
(Proposition~\ref{limit-ultra}).

As the Stone space of a Boolean algebra is metrizable iff the Boolean algebra is at most countable,
the space $(\SL,\tau)$ is metrizable iff $L$ is at most countable.
For an arbitrary $L$, we define a natural ultrametric on $S_L$
as follows.
For two minimal $L$-structures $M$ and $N$, the distance
between their quantifier-free theories is defined to be
equal to $1/m$, where $m$ is maximal with the property that
$M$ and $N$ satisfy the same atomic $L$-sentences of length at most $m$.
We study the properties of that ultrametric and its relation with the Stone topology on $\SL$.
We show that the topology  of the ultrametric space $\SL$ is finer or equal than the Stone topology  on $\SL$;
the two topologies coincide iff
the signature $L$ is locally finite.
(We call $L$ locally finite if $L$ contains finitely many $n$-ary symbols for any $n$.)
In particular, the ultrametric space $S_L$ is compact iff
$L$ is locally finite (Theorem~\ref{compact-S}).

As an application of compactness of the Stone space of minimal structures
we give a proof of compactness of certain subclasses
in the Gromov--Hausdorff space of metric spaces
(Theorem~\ref{nu-bound}, Corollary~\ref{compact}).
This refines the known result based on Gromov's ideas~\cite{BBI, Gro2}
that any uniformly totally bounded class of compact metric spaces is precompact
in the Gromov--Hausdorff topology.
For the proof, we associate with every semi-metric space certain relational structures
with the same uiniverse called semi-metric structures; the class of such structures is shown to
be universally axiomatizable.

For basics of model theory, see~\cite{H}.
The facts and notions of metric geometry we need  can be found in~\cite{BBI}.

\section{Minimal structures}

Let $L$ be a signature containing at least one constant symbol;
in this case the set $\TL$ of ground $L$-terms (that is, the terms without
free variables) is not empty.
We call an $L$-structure \emph{minimal} if it has no proper substructures, or,
equivalently, is generated by the empty set. Clearly,
an $L$\nobreakdash-structure is minimal iff
any its element is the value of some ground $L$-term in the structure.
For any $L$-structure $M$ the substructure generated by the empty set is
a unique minimal substructure; we call it the \emph{core of $M$}
and denote by $\core(M)$. We denote the class of all minimal $L$-structures by $\ML$.

Let $L_0$ be an arbitrary signature,
$C$ a nonempty set of constant symbols disjoint with $L_0$, and $L=L_0(C)$.
Clearly, an $L$-structure $M$ is minimal if and only if
the set $\{c^M: c\in C\}$ generates its $L_0$-reduct $M_0$.
Thus, any structure becomes minimal after naming its generators.
We call minimal $L$\nobreakdash-structures \emph{$C$-marked $L_0$\nobreakdash-structures}.
For any $L$-structure $M$ its core is a $C$-marked $L_0$-structure ---
it is the minimal substructure generated by $\{c^M: c\in C\}$.

The notion of marked structure generalizes the notion of marked group
(see \cite{ChG}),
which is defined to be a group with distinguished generators
(not necessarily all distinct). In this case $L_0=\{\cdot\,,^{-1}, e\}$,
and $C$ consists of names of generators of the group.
Note that here we do not assume that the group is finitely generated, and $C$ is finite.
If $C$ is finite, $|C|=n$, then $C$-marked groups are called $n$-marked groups.

Let $\QFL$ be the set of all quantifier-free $L$-sentences.
For an $L$-structure $M$
we denote by $\qf(M)$
the quantifier-free theory of $M$, that is,
the set of sentences in $\QFL$
that hold in $M$, and by
$\at(M)$ the set consisting of all atomic or
negated atomic $L$-sentences from $\qf(M)$.

We will need the following essentially known facts.

\begin{proposition}\label{qf}
For minimal $L$-structures $M$ and $N$ the following are equivalent:
\begin{enumerate}
\item
$M\simeq N$;
\item
$\qf(M)=\qf(N)$;
\item
$\at(M)=\at(N)$.
\end{enumerate}
\end{proposition}

\begin{proof}
$(1)\Rightarrow(2)\Rightarrow(3)$ is obvious.
If (3) then the map $t^M\mapsto t^N$ is a well-defined isomorphism
from $M$ onto $N$, and so (1).
\end{proof}

Due to this fact, we call $\qf(M)$ the \emph{isomorphism type} of a minimal $L$-structure~$M$.

\begin{proposition}\label{max}
For $S\subseteq \QFL$ the following are equi\-valent:
\begin{enumerate}
\item
$S=\qf(M)$, for some minimal $L$-structure $M$,
\item
$S$ is a maximal finitely satisfiable subset of $\QFL$;
\item
$S$ is finitely satisfiable, and
for any $\phi\in\QFL$
either $\phi\in S$ or $\neg\phi\in S$.
\end{enumerate}
\end{proposition}

\begin{proof}
$(1)\Rightarrow(2)\Rightarrow (3)$ is easy; we prove $(3)\Rightarrow(1)$.
By (3), $t=s\in S$ is an equivalence relation on $\TL$.
Denote by $[t]$ the equivalence class of $t\in\TL$.
Let $M$ be the $L$\nobreakdash-structure whose universe is $\{[t]: t\in \TL\}$,
and
$$f^M([t_1],\dots,[t_n])=[f(t_1,\dots,t_n)],$$
$$R^M=\{([t_1],\dots,[t_n]): R(t_1,\dots,t_n)\in S\},$$
for any function $L$-symbol $f$ and relation $L$-symbol $R$ of arity $n$.
Due to (3), $f^M$ and $R^M$ are well-defined.
By induction, $t^M=[t]$, for any $t\in\TL$.
Then $t=s\in S$ iff $t^M=s^M$, for any $t, s\in\TL$.
Using (3), it is easy to show by induction that
$\phi\in S$ iff $M\models\phi$, for any $\phi\in\QFL$. Thus $S=\qf(M)$.
Since $[t]$ is $t^M$ for any $t\in\TL$, the structure $M$ is minimal.
\end{proof}

Since, by Zorn's lemma, any finitely satisfiable subset of $\QFL$ can be completed to a maximal such subset,
we have

\begin{corollary}[Herbrand's theorem]\label{herbrand}
Any finitely satisfiable subset of $\QFL$ has a minimal model.
\end{corollary}

\begin{remark}
\emph{Herbrand's theorem
is a weak version of model-theoretic compactness theorem. This version admits a simple proof given above,
and in the present paper we need only this version of compactness theorem.}
\end{remark}

Denote by $\SL$ the set of all maximal finitely satisfiable subsets of $\QFL$.
Due to Proposition~\ref{max}, this is the set of
isomorphism types of minimal $L$-structures.

\section{The Stone space of minimal structures}\label{stone}

%\subsection{Topology on the set of isomorphism types of minimal structures}
\subsection{Topology on $\SL$}
It is easy to see that, for any $S\in\SL$ and $\phi, \psi\in\QFL$,
\begin{enumerate}
\item
$\phi\wedge\psi\in S\ \text{iff}\ \phi\in S\ \text{and}\ \psi\in S;$
\item
$\phi\vee\psi\in S\ \text{iff}\ \phi\in S\ \text{or}\ \psi\in S;$
\item
$\neg\phi\in S\ \text{iff}\ \phi\notin S$.
\end{enumerate}
In other words, if $U_\phi=\{S\in\SL: \phi\in S\}$, we have
$$
(1)\ U_{\phi\wedge\psi}=U_{\phi}\cap U_{\psi};\qquad
(2)\ U_{\phi\vee\psi}=U_{\phi}\cup U_{\psi};\qquad
(3)\ U_{\neg\phi}=U_{\phi}^c.
$$
Due to (1), $\{U_\phi: \phi\in\QFL\}$ is a basis of a topology on $\SL$;
we denote the topology by $\tau$.
Due to (3), the sets $U_\phi$ are clopen in $\tau$.
It is easy to show that $U_\phi=U_\psi$ iff $\phi$ and $\psi$ are equivalent.

Let $T$ be the set of finite conjunctions of atomic or negated atomic $L$-sentences.
Since any $\phi\in\QFL$ is equivalent to a finite disjunction of sentences from $T$
then, due to (2), $\{U_\phi: \phi\in T\}$ is a basis of $\tau$ as well.

\begin{proposition}\label{st-comp}
The topological space $(\SL,\tau)$  is
\begin{enumerate}
\item[(i)]
totally disconnected, and
\item[(ii)]
compact.
\end{enumerate}
\end{proposition}

\begin{proof}
(i) Let $S$ and $P$ be different elements of $\SL$. Let, say,
$\phi\in S$ and $\phi\notin P$. Then $S\in U_\phi$, and $P\in U_{\phi}^c$.
Since $U_\phi^c=U_{\neg\phi}$, by (3), both $U_\phi$ and $U_\phi^c$ are open, and the result follows.

(ii)
Suppose $\{U_\phi: \phi\in T\}$ covers $\SL$, where $T\subseteq\QFL$.
Then $\bigcap_{\phi\in T}U_{\neg\phi}=\emptyset$, that is,
there is no $S\in \SL$
with $\{\neg\phi :\phi\in T\}\subseteq S$.
Then, for some finite $F\subseteq T$, there is no $S\in \SL$
with $\{\neg\phi :\phi\in F\}\subseteq S$;
otherwise $\{\neg\phi :\phi\in T\}$ would be finitely satisfiable, and so could be
completed to a member of $\SL$, by Zorn's lemma.
Hence  $\bigcap_{\phi\in F}U_{\neg\phi}=\emptyset$, and so
$\{U_\phi: \phi\in F\}$ covers $\SL$.
\end{proof}

\begin{remark}
\emph{The proof of compactness of the topology $\tau$ did not use the model-theoretic compactness theorem
even in its weaker Herbrand's version.}
\end{remark}

\begin{proposition}\label{clopen-fin-ax}
Any set clopen in $\tau$ is $U_\phi$, for some $\phi\in\QFL$.
\end{proposition}

\begin{proof}
Any set $U$ open in $\tau$ is $\bigcup_{\phi\in T}U_\phi$, for some $T\subseteq\QFL$.
If $U$ is closed, it is compact, by Proposition~\ref{st-comp}(2), and hence
$U=\bigcup_{\phi\in F}U_\phi$,
for some finite $F\subseteq T$.
By (2), $U=U_\phi$, where $\psi=\bigvee_{\phi\in F}\phi$.
\end{proof}

For $\phi\in\QFL$, denote by $[\phi]$ the set of all $\psi\in\QFL$ equivalent to $\phi$.
The sets $[\phi]$ form a Boolean algebra with the operations induced by the logical operators
$\wedge$, $\vee$, and $\neg$.
We denote that Boolean algebra by  $\QFLb$, and its Stone space by  $\St$.

Recall that for a Boolean algebra $\mathscr{B}$ its Stone space ${\rm St}(\mathscr{B})$ is defined to be
the topological space
whose points are ultrafilters of $\mathscr{B}$, and a basis of topology
is $\{U_b: b\in\mathscr{B}\}$, where
$$U_b=\{p: p\ \text{is an ultrafilter of $\mathscr{B}$ with}\ b\in p\}.$$
It is known (see~\cite[\S 8]{S}) that ${\rm St}(\mathscr{B})$
is compact and totally disconnected;
it is metrizable iff it has a countable basis iff $|\mathscr{B}|\le\aleph_0$;
its clopen sets are exactly the sets $U_b$.
Any closed subspace $X$ of ${\rm St}(\mathscr{B})$ is a compact, totally disconnected space;
its clopen sets are exactly the sets $U_b\cap X$, and they form a basis of $X$.

For $T\subseteq\QFL$ denote $\{[\phi]: \phi\in T\}$ by $[T]$.
It is not difficult to show that  $S\mapsto [S]$ is a bijection between $\SL$ and
the set of ultrafilters of $\QFLb$. Moreover, $[U_\phi]=U_{[\phi]}$, for any $\phi\in\QFL$.
Therefore $S\mapsto [S]$ is a natural homeomorphism between the topological space $(\SL,\tau)$ and
the Stone space ${\rm St}(\QFLb)$.
Because of that, we call $\tau$ the \emph{Stone topology on $\SL$.}
Since $\SL$ is the set of isomorphism types of minimal $L$-structures,
we call the topological space $(\SL,\tau)$ the \emph{Stone space of isomorphism types of
minimal $L$-structures}, or,
for short, the Stone space $\SL$.

As $|\QFLb|\le\aleph_0$ iff $|L|\le\aleph_0$,
the Stone space $\St$ is metrizable iff $|L|\le\aleph_0$.
Since any compact metric space is separable, $\St$ is separable if $|L|\le\aleph_0$.
Thus, the Stone space of minimal $L$-structures is metrizable and separable if $|L|\le\aleph_0$.

For an $L$-sentense $\phi$ denote by
$\ModM(\phi)$ the class of minimal models of $\phi$, and by
$\ModS(\phi)$ the set of isomorphism types of minimal models of~$\phi$.
In other words,
$$\ModM(\phi)=\{M\in\ML: M\models\phi\},$$
$$\ModS(\phi)=\{\qf(M): M\in\ML,\ M\models\phi\}.$$
Clearly, for $\phi\in\QFL$,
$$\ModS(\phi)=U_\phi.$$
Thus for any $\phi\in\QFL$ the set $\ModS(\phi)$ is a clopen subspace of the Stone space~$\SL$.

\begin{proposition}\label{exist-open}
If $\phi$ is an existential $L$-sentence then
$\ModS(\phi)$ is open.
\end{proposition}

\begin{proof}
Let $\phi$ be $\exists v_1\dots v_n\psi(v_1,\dots,v_n),$
where $\psi$ is quantifier-free.
Clearly, $\phi$ holds in a minimal $L$-structure $M$ iff $M\models\psi(t_1,\dots,t_n)$
for some $t_1,\dots,t_n\in\TL$.
Therefore $\ModS(\phi)$
is the union of all sets $\ModS(\psi(t_1,\dots,t_n))$, where $t_1,\dots,t_n\in\TL$.
Since all $\ModS(\psi(t_1,\dots,t_n))$ are clopen,
$\ModS(\phi)$ is open.
\end{proof}

\begin{proposition}\label{univ-closed}
If $\phi$ is a universal $L$-sentence then
$\ModS(\phi)$ is closed.
\end{proposition}

\begin{proof}
The sentence $\phi$ is equivalent to $\neg\theta$ for some existential $L$-sentence $\theta$. Then
the complement of $\ModS(\phi)$ in $\SL$ is the set $\ModS(\theta)$, which is open
by Proposition~\ref{exist-open}.
\end{proof}

For an $L$-theory $T$, denote by $\ModS(T)$ the set of isomorphism types of minimal
models of $T$.

\begin{proposition}\label{univ-theory}
If $T$ is a universal $L$-theory then $\ModS(T)$ is closed.
\end{proposition}

\begin{proof}
Since $\ModS(T)=\bigcap_{\phi\in T}\ModS(\phi),$
this follows from Proposition~\ref{univ-closed}.
\end{proof}

Similarly to the Stone topology on $\SL$, one can define a topology on the class $\ML$
whose basis consists of the classes $\ModM(\phi)$, where $\phi\in\QFL$.
We call that topology the \emph{Stone topology on} $\ML$.
The class $\ML$ equipped with that topology is called the \emph{Stone space
of minimal $L$-structures}, or,
for short, the Stone space~$\ML$.
Obviously, analogs of Propositions~\ref{st-comp}--\ref{univ-theory}
hold for it, with one exception: the Stone space $\ML$ is not Hausdorff (and so not totally disconnected),
because any isomorphic but different members of $\ML$
cannot be separated by open sets.
Note that compactness of the Stone space $\ML$ is based on Herbrand's theorem.

If $L=L_0(C)$, we call the Stone space $\SL$ the \emph{Stone space
of isomorphism types $C$-marked $L_0$-structures}.
Let $L_0=\{\cdot\,,^{-1}, e\}$, and $\gamma$ be the universal $L_0$-sentence
that axiomatizes the class of groups.
Then $\ModS(\gamma)$ is a closed subspace of the Stone space $\SL$,
by Proposition~\ref{univ-theory}.
Its points are isomorphism types of groups with generators
marked by elements of $C$.
We call this topological space the \emph{space of isomorphism types of $C$-marked groups} and denote it by $\GC$.
The space $\GC$ is compact and totally disconnected.

For $\psi\in\QFL$, the set $U_{\psi}\cap\GC$ is
the set of isomorphism types
of $C$-marked groups satisfying $\psi$; it is clopen in $\GC$.
Any clopen set in $\GC$ is of that form, and the sets
$U_{\psi}\cap\GC$ form a basis of $\GC$.
Moreover, for the set $\Psi$ of
finite conjunctions of $L$-sentences of the form $w=e$ or $w\ne e$, where $w$ is a group word over $C$,
the set $\{U_{\psi}\cap\GC: \psi\in\Psi\}$ is a basis of the space $\GC$.

For a finite set of constant symbols $C$ with $|C|=n$, the space $\GC$
is exactly the space of isomorphism types of $n$-marked groups introduced in~\cite{Gri};
we denote it by~$\Gn$.

\begin{proposition}\label{GC-clopen}
The set $\GC$ is clopen in $\SL$
iff $C$ is finite.
\end{proposition}

\begin{proof}
Suppose $C$ is finite.
Let $\Theta$ be the set of quantifier-free $L$-sentences
$$(a\cdot b)\cdot c=a\cdot(b\cdot c),\qquad c\cdot e=e\cdot c=c,\qquad c\cdot c^{-1}=c^{-1}\cdot c=e$$
for all constant symbols $a,b,c$ in $C$. Clearly, $\Theta$ is finite.
It is easy to show that
$\GC=\ModS(\Theta)=U_\theta,$
where $\theta=\bigwedge\Theta$.
Therefore $\GC$ is clopen.

Now we show that if $C$ is infinite then $\GC$ is not clopen.
Suppose not, and $\GC=U_\theta$, where $\theta\in\QFL$.
Let $C^*$ be the finite set of all $c\in C$ that occurs in $\theta$.
Consider any $C^*$-marked group $M^*$. It is easy to construct
a minimal $L$-structure $N$
such that $M^*$ is a substructure of its $L_0(C^*)$-reduct,
and the $L_0$-reduct of $N$ is not a group.
Since any $L$-expansion of $M^*$ belongs to $\GC$, we have $M^*\models\theta$.
Therefore $N\models\theta$, and hence $\qf(N)\in\GC$.  Contradiction.
\end{proof}

\begin{remark}
\emph{A special case of Proposition~\ref{univ-closed} was proven in \cite[Section 5.2]{ChG}:
\emph{for any universal sentence $\theta$ in the group language,
the set of isomorphism types of  $n$-marked groups satisfying $\theta$ is closed in $\Gn$.}
This fact is slightly weaker than Proposition~\ref{univ-closed}: for example, it does not not imply
closedness of the set $\K$ of isomorphism types of $n$\nobreakdash-marked centerless groups, because
the class of centerless groups is not closed under subgroups and therefore
is not universally axiomatizable. However, Proposition~\ref{univ-closed} implies
that $\K$ is closed in $\Gn$, because for any finite $C$
the class of $C$-marked centerless groups is axiomatizable by the universal sentence
$$\forall v(\bigwedge_{c\in C}[v,c]=1\to v=1).$$
Note that $\K$
is not open in $\Gn$ if $n=|C|>1$.
Indeed, let $G$ be a free group of rank $n$,
and $N_k$ a free $k$-nilpotent group of rank $n$;
then $N_k\simeq G/G_k$, where $G_k$ is the $k$-th member of the lower central
series of $G$.
Consider $G$ and $N_k$ as groups with marked free generators.
Then $G$ is a limit of the sequence $N_1, N_2,\dots$;
this follows from the well-known fact that $\bigcap_{k=1}^\infty G_k=1$.
But $G$ is centerless, and all $N_k$ are not.}
\end{remark}

For a variety $\V$ of $L$-structures,
we call a $\V$-free structure generated by the empty set
a $\V$-free minimal structure.
Denote by $\V_\star$ the set of isomorphism types of
minimal $L$-structures from $\V$.

\begin{proposition}\label{fp}
Let $\V$ and $\W$ be varieties of $L$\nobreakdash-struc\-tures, and $\V\subseteq\W$.
The following are equivalent:
\begin{enumerate}
\item
$\V_\star$ is clopen in $\W_\star$;
\item
the $\V$-free minimal structure $N$
is finitely presentable in  $\W$.
\end{enumerate}
\end{proposition}

\begin{proof}
$(2)\Rightarrow(1).$
Suppose $N$ is finitely presented in $\W$
by atomic $L$-sentences $\phi_1,\dots,\phi_n$.
Then $$\V_\star=\ModS(\phi)\cap\W_\star,$$  where $\phi=\bigwedge_i\phi_i$.
So $V_\star$ is clopen in $\W_\star$.

$(2)\Rightarrow(1).$
Suppose $\V_\star$ is clopen in $\W_\star$.
Since $\W_\star$ is closed by Proposition~\ref{univ-theory},
$\V_\star=\ModS(\phi)\cap\W_\star,$
for some $\phi\in\QFL$.
We may assume that $\phi$
is a finite disjunction of sentences of the form
$$\phi_1\wedge\dots\wedge\phi_n\wedge
\neg\psi_1\wedge\dots\wedge\neg\psi_k,$$
where all $\phi_i, \psi_j$ are atomic $L$-sentences.
Then one of these disjuncts --- say, the disjunct written above--- holds in $N$.
Let $M$ be the minimal $L$-structure
presented in $\W$ by the relations $\phi_1,\dots,\phi_n$.
Then there is a homomorphism from $M$ onto $N$.
Hence all $\neg\psi_i$ hold in $M$. Therefore $\phi$ holds in $M$,
and so $M\in \V$.
Since $N$ is $\V$-free there is a homomorphism from $N$ onto $M$.
Hence this homomorphism is an isomorphism.
Thus $N$ is finitely presented in $\W$.
\end{proof}

\begin{corollary}
Let $\V$ be a group variety, and $n\ge 1$. Then the set of isomorphism types of $n$-marked $\V$-groups
is clopen in $\Gn$ iff the $\V$-free group of rank $n$ is finitely presented.
\end{corollary}

For example, if $\V$ is any nilpotent group variety then the class of $n$-marked
$\V$\nobreakdash-groups is clopen in $\Gn$. Since for $n,m\ge 2$ the $n$-generated free $m$-solvable
group is not finitely presented~\cite{Shm}, the class of $n$-marked $m$-solvable
groups is not open in $\Gn$.
The latter fact was explained in \cite[Section 2.6]{ChG} in a completely different way
based on some D.~V.~Osin's result.
Note that there is an open question posed by A.~Yu.~Olshanski
whether any finitely presented relatively free group is virtually nilpotent.

\subsection{Compactness of $\Gn$ and Gromov's theorem}
Now we explain how one can use compactness of  $\Gn$
to deduce from Gromov's theorem its `bounded' version formulated at the beginning of
the present paper.

Fix $n$, $k$, and $d$.
Let
$L_0=\{\cdot,^{-1},e\}$, and $L=L_0(C)$, where
$C=\{c_1,\dots,c_n\}$.

It is easy to construct $\sigma_m\in\QFL$
which says about a $C$-marked group that for all $r=1,\dots,m$
the size of the ball of radius $r$ centered at the identity is $\le kr^d$.
Also, it is not difficult to construct $\tau_m\in\QFL$
which says about a $C$-marked group that
it has a nilpotent subgroup of class $\le m$ and index $\le m$ (see \cite[Section~7]{DW}).

Let $\phi_m$ denote $\sigma_m\to\tau_m$.
It is easy to see that if $m<l$ then $\phi_m$ implies $\phi_l$.
Every $C$-marked group $M$ satisfies $\phi_m$
for some~$m$ (possibly, depending on $M$).
Indeed, if $M$ is virtually nilpotent then $M\models\tau_m$ for some $m$;
if $M$ is not virtually nilpotent then, by Gromov's theorem, $M$ is not of polynomial growth,
and therefore $M\nvDash~\sigma_m$, for some $m$.

Let $\K_m$ denote $U_{\phi_m}\cap\GC$.
Then
$\K_m$ is clopen in $\GC$.
Thus $\{\K_m: m\ge 1\}$ is an open cover of $\GC$. It has a finite subcover because
$\GC$ is compact.
Since $\K_m\subseteq \K_{m+1}$ for all $m$,
we have
$\GC=\K_m$ for some $m$.
Thus there is $m$ such that every $C$-marked group satisfies $\sigma_m\to\tau_m$,
and the result follows.

\begin{remark}
\emph{Note that the proof above is based on compactness of the Stone space of $C$-marked groups,
which follows from a general fact on compactness of Stone spaces of Boolean algebras.
For a proof of the latter fact one needs only Zorn's lemma but not
model-theoretic compactness theorem.
The proof of the result given in \cite[Section 7]{DW} is based on model-theoretic compactness theorem;
so our proof is different, even though uses the same idea.}
\end{remark}

Another way to realize that idea is to use ultraproducts.
Towards a contradiction,  suppose for every $i$ there is a $C$-marked group $M_i$ with
$M_i\models\neg\phi_i$.
If $j>i$ then $M_j\models~\neg\phi_i$ because
$\phi_i$ implies $\phi_j$.
Then, by the \L o\'{s} theorem, for any non-principal ultraproduct $M$ of the $C$-marked groups $M_i$
we have $M\models\neg\phi_j$, for all $j$.
Then all $\phi_j$ fail in the $C$-marked group $\core(M)$, contrary to Gromov's theorem.
\bigskip

\subsection{Topology on $\ML$ and ultraproducts}
In general, there is a link between ultraproducts and the Stone topology on the class of minimal structures
(cf. \cite[Proposition~6.4]{ChG}, where a link
between ultraproducts and convergence of groups in the space of marked groups had been
demonstrated).

\begin{proposition}\label{clo-ultra}
Let $X$ be a subset of $\ML$, and $M\in\ML$. Then the following are equivalent:
\begin{enumerate}
\item
$M$ belongs to the closure of  $X$ in the Stone space $\ML$;
\item
$M$ is isomorphic to the core of an ultraproduct of structures from $X$;
\item
$M$ is embeddable into an ultraproduct of structures from $X$.

\end{enumerate}
\end{proposition}

\begin{proof}
Obviously, $(2)\Rightarrow(3)$.

$(3)\Rightarrow(1)$.
Suppose $M$ is embeddable into an ultraproduct $\prod_{i\in I}M_i/D$
of structures from $X$.
We show that any  basic neighbourhood  $\ModM(\phi)$
of $M$, where $\phi\in\QFL$, contains an element of $X$.
Since $\phi$ is quantifier-free and holds in $M$,
it holds in the ultraproduct.
Therefore,
by  the \L o\'{s} theorem,
$$I_\phi=\{i\in I: M_i\models\phi\}\in D.$$
Hence $I_\phi\ne\emptyset$, and so $\ModM(\phi)$
contains an element of $X$.
\medskip

$(1)\Rightarrow(2).$ Let $X=\{M_i: i\in I\}$.
For $\phi\in\qf(M)$ denote
$$I_\phi=\{i\in I: M_i\models\phi\};$$
then $I_\phi\ne\emptyset$,
because $X\cap U_\phi\ne\emptyset$, by~(1).
The set $$P=\{I_\phi:\phi\in\qf(M)\}$$ is closed under finite intersections,
because if $\phi_1,\dots,\phi_n\in\qf(M)$ then
$$I_{\phi_1}\cap\dots\cap I_{\phi_n}= I_{\phi_1\wedge\dots\wedge\phi_n},\quad\text{and}\quad
\phi_1\wedge\dots\wedge\phi_n\in\qf(M).$$
Therefore $P$ has the finite intersection
property, and hence can be completed to an ultrafilter $D$ on $I$.
For any $\phi\in \qf(M)$ we have $I_\phi\in D$, and hence $\prod_{i\in I}M_i/D\models\phi$,
by  the \L o\'{s} theorem.
It follows that any $\phi\in \qf(M)$ holds in the core of the ultra\-product. Therefore $M$ is
isomorphic to the core,
by Proposition~\ref{qf}.
\end{proof}

A point $M$ of the Stone space $\ML$ is called a \emph{limit point} of a subset $X$ of $\ML$
if every open neighbourhood of $M$ in $\ML$ contains a member of $X$ which is non-isomorphic to $M$.

\begin{proposition}\label{limit-ultra}
Let $X$ be a subset of $\ML$, and
$M$ a structure in $\ML$, which is non-isomorphic to any member of $X$.
Then the  following are equivalent:
\begin{enumerate}
\item
$M$ is a limit point of $X$ in the Stone space $\ML$;
\item
$M$ is isomorphic to the core of a non-principal ultraproduct of pairwise non-isomorphic structures from $X$;
\item
$M$ is embeddable into a non-principal ultraproduct of pairwise non-isomorphic structures from $X$.

\end{enumerate}
\end{proposition}

\begin{proof}
Obviously, $(2)\Rightarrow(3)$.

$(3)\Rightarrow(1)$.
Suppose $M$ is embeddable into
$\prod_{i\in I}M_i/D$, where $\{M_i:i\in I\}$ is a family of pairwise non-isomorphic
structures from $X$, and $D$ is a non-principal ultrafilter on $I$.
We need to show that any  basic neighbourhood $\ModM(\phi)$ of $M$ contains an element of $X$
non-isomorphic to $M$.
Since $\phi$ is quantifier-free and holds in $M$,
it holds in the ultraproduct.
Therefore,
by  the \L o\'{s} theorem,
$$I_\phi=\{i\in I: M_i\models\phi\}\in D.$$
Since the ultrafilter $D$ is non-principal, $|I_\phi|>1$.
Since all $M_i$ are pairwise non-isomorphic, there is $j\in I_\phi$ such that
$M_j$ is not isomorphic to $M$. Then
$M_j\in X$, and $M_j\in\ModM(\phi)$.

$(1)\Rightarrow(2)$.
Let $\{M_i: i\in I\}$ be a family of representatives of all isomorphism types of structures in $X$,
which are not isomorphic to $M$. For any $\phi\in\qf(M)$, the set
$$I_\phi =\{i\in I: M_i\models\phi\}$$
is infinite. Indeed, suppose not.
By Proposition~\ref{qf},
for each $i$ there is $\theta_i\in\qf(M)$ such that $M_i\nvDash\theta_i$.
Since $M$ is a limit point of $X$, there is $N\in X$ which is
non-isomorphic to $M$ and such that $N\models\bigwedge_i\theta_i$.
Then none of $M_i$ is isomorphic to $N$. Contradiction.
The set $$P=\{I_\phi:\phi\in\qf(M)\}$$ is closed under finite intersections,
as in the proof of  $(1)\Rightarrow(2)$ at Proposition~\ref{clo-ultra}.
Let $F$ be the Fr\'{e}chet filter on $I$.
The set $P\cup F$ has the finite intersection property:
otherwise,  for some $\phi\in\qf(M)$
the set $I_\phi$ is disjoint with a set from $F$, and hence is finite.
Hence $P\cup F$ is contained in an ultrafilter $D$ on $I$.
The ultrafilter $D$ is non-principal because it contains $F$.
For any $\phi\in\qf(M)$ we have $I_\phi\in D$, and therefore $\prod_{i\in I}M_i/D\models\phi$,
by  the \L o\'{s} theorem.
It follows that any $\phi\in\qf(M)$ holds in the core of the ultraproduct. Therefore $M$ is
isomorphic to the core, by Proposition~\ref{qf}.
\end{proof}

\section{The ultrametric space of minimal structures}\label{ultrametric}

For $m\ge 1$,
we say that $L$-structures $M$ and $N$ are $m$\nobreakdash-\emph{close} if
$$M\models\theta\Leftrightarrow N\models\theta,$$
for any atomic $L$-sentence $\theta$ of length $\le m$.

Note that minimal $L$-structures $M$ and $N$ are $m$-close for arbitrary large $m$
iff $\at(M)=\at(N)$ iff $M\simeq N$, by Proposition~\ref{qf}.

For minimal $L$-structures $M$ and $N$ we define $d(M, N)$,
the distance between $M$ and $N$, as  follows.
If $M\simeq N$, put $d(M,N)=0$. Otherwise
$d(M,N)$ is defined to be $1/m$,
where $m$ be the maximal positive integer such that $M$ and $N$ are $m$-close.

It is easy to see that
$d$ is \emph{semi-ultrametric} on $\ML$, that is,
for any $M,N,Q\in\ML$
\begin{enumerate}
\item
$d(M,N)\ge 0$, and $d(M,M)=0$;
\item
$d(M,N)=d(M,N)$;
\item
$d(M,P)\le\max\{d(M,N), d(N,Q)\}$.
\end{enumerate}

Since $d(M,N)=0$ iff $M\simeq N$, the semi-ultrametric $d$ induces an ultrametric on
the set of isomorphism types of minimal $L$-structures, that is, on $\SL$.
We denote the induced ultrametric on $\SL$ by the same letter $d$;
so for any $S,P\in\SL$, we have
$d(S,P)=d(M,N)$, where $S=\qf(M)$ and $P=\qf(N)$.

Clearly, for any $S,P\in \SL$,  we have
\begin{enumerate}
\item[(i)]
$d(S,P)\in\{1/m: m\ge 1\}\cup\{0\}$, and
\item[(ii)]
$d(S,P)\le 1/m$ means that
$\theta\in S$ iff $\theta\in P$,
for any atomic $L$-sentence $\theta$ of length $\le m$.
\end{enumerate}

We call  $(\ML,d)$ and
$(\SL,d)$ the \emph{semi-ultrametric} and \emph{ultrametric space of minimal $L$-structures},
respectively.

Clearly, in $\ML$ and $\SL$ for any point $x$ the open ball $B(x,\varepsilon)$ is the whole space if
$\varepsilon>1$. If for a positive integer $m$
$$1/(m+1)<\varepsilon\le 1/m,$$
then the open ball $B(x,\varepsilon)$ is equal to the closed ball $\bar{B}(x,1/m)$.
Thus in the spaces $\ML$ and $\SL$ any open ball is a closed set. It follows that
\emph{the metric space $\SL$ is totally disconnected}.

\begin{proposition}\label{clopen}
For any $\phi\in\QFL$,
the set\ \ $\ModS(\phi)$ is clopen in the ultrametric space $\SL$.
\end{proposition}

\begin{proof}
Since a boolean combination of clopen sets is clopen, we may assume that $\phi$ is atomic.
Let $m$ be the length of $\phi$.
Denote $\ModS(\varphi)$ by $U$.
For any $S\in\ML$, if $\phi\in S$
then
$B(S,1/m)\subseteq U,$
and if
$\phi\notin S$ then $B(S,1/m)\subseteq U^c.$
So $U$ is clopen.
\end{proof}

Since $\{\ModS(\phi): \phi\in\QFL\}$ is a basis of the Stone topology on $S_L$, we have

\begin{corollary}\label{finer}
The ultrametric topology is equal to or finer than
the Stone topology on~$S_L$.
\end{corollary}

In general, the two topologies do not coincide:
in the Stone space $\SL$ the clopen sets are exactly $\ModS(\phi)$, where $\phi\in\QFL$, but
in the ultrametric space $\SL$ it is not always so.
For example, in Proposition~\ref{GC-clopen} we proved
that if $C$ is infinite then $\GC\ne\ModS(\phi)$, for any $\phi\in\QFL$.
However,

\begin{proposition}
The set $\GC$ is clopen in the ultrametric space $S_L$, for any $C$.
\end{proposition}

\begin{proof}
Let $\Theta$ be defined as in the proof of Proposition~\ref{GC-clopen}; then
$$\GC=\ModS(\Theta)=\bigcap\{\ModS(\theta):\theta\in\Theta\}.$$
Since all $\ModS(\theta)$ are clopen, $\GC$ is closed.
Also, $\GC$ is open because if $S\in\GC$, and $m$ is the maximal length
of sentences in $\Theta$, then $B(S,1/m)\subseteq\GC$.
\end{proof}

We call a signature $L$ \emph{locally finite} if
for every $n$ the set of $n$-ary symbols in $L$ is finite.
Clearly, any locally finite signature is finite or countable.

\begin{theorem}\label{compact-S}
The following are equivalent:
\begin{enumerate}
\item
any clopen set in the ultrametric space $\SL$ is $\ModS(\phi)$ for some $\phi~\in~\QFL$;
\item
any open ball in the ultrametric space $\SL$ is $\ModS(\phi)$ for some $\phi\in\QFL$;
\item
the Stone and ultrametric topologies on $\SL$ coincide;
\item
the ultrametric space $\SL$ is compact;
\item
the ultrametric space $\SL$ is separable;
\item
the signature $L$ is locally finite.
\end{enumerate}
\end{theorem}

\begin{proof}
We prove $(1)\Rightarrow (2)\Rightarrow (3)\Rightarrow (4)\Rightarrow (5)\Rightarrow (6)\Rightarrow (2)$ and
$(2)\wedge(4)\Rightarrow (1)$.

$(1)\Rightarrow (2)$ because every open ball in $\SL$ is closed.

$(2)\Rightarrow(3)$. Due to (2), every open set in the ultrametric space $\SL$
is open in the Stone topology on $\SL$. Together with Corollary~\ref{finer}, this gives (3).

$(3)\Rightarrow(4)$ because the Stone space $S_L$ is compact, which is a consequence of Herbrand's theorem.

$(4)\Rightarrow(5)$ because any compact metric space is separable.

$(5)\Rightarrow(6)$.
Suppose there are infinitely many $n$-ary symbols in $L$.
We show that $\SL$ is not separable.

First we show that there is a family $\{\theta_i: i<\omega\}$ of atomic $L$-sentences
of the same length $m$ such that for
any $I\subseteq\omega$ the set of sentences
$$\Theta_I=\{\theta_i: i\in I\}\cup\{\neg\theta_i: i\notin I\}$$
holds in some  minimal $L$-structure $N_I$.

Let $c$ be a constant symbol in $L$.
If $L$ contains infinitely many distinct and different from $c$ constant symbols $c_0,c_1,\dots$,
one can take the sentence $c_i=c$\ as~$\theta_i$.
If $L$ contains infinitely many distinct
$n$-ary function
symbols $f_0,f_1,\dots$, where $n\ge 1$,
one can take the sentence $f_i(c,\dots,c)=c$ as~$\theta_i$.
If $L$ contains infinitely many distinct $n$-ary
relation symbols $P_0,P_1,\dots$
one can take the sentence $P_i(c,\dots,c)$ as~$\theta_i$.
Clearly, for such choice of $\theta_i$ the set $\Theta_I$ holds in some $L$-structure,
and hence in its core $N_I$.

We prove that no countable  subset is dense in $\SL$.
To show that, we construct
for any sequence $(M_i: i<\omega)$ in $\ML$ a member of $\ML$
which is not $m$-close to $M_i$ for every $i<\omega$.
Let $$I=\{i<\omega: M_i\nvDash\theta_i\}.$$
Then for any $i$ the
structure $N_I$ is not $m$-close
to $M_i$ because $N_i\models\theta_i$ but $M_i\nvDash\theta_i$.

$(6)\Rightarrow(2).$
Let $S\in\SL$, and $m\ge 1$. We show that
$B(S,1/m)=\ModS(\phi)$ for some $\phi\in\QFL$.
Since $L$ is locally finite, the set of atomic
$L$\nobreakdash-sentences of length $\le m+1$ is finite.
Let $\phi$ be the conjunction of all sentences from $\at(M)$ of length $\le m+1$.
Then $P\in \ModS(\phi)$ means exactly that $P$ and $S$ are $(m+1)$-close,
that is, $P\in B(S,1/m)$.

$(2)\wedge(4)\Rightarrow(1).$
Let $U$ be a clopen set in the ultrametric space $\SL$. Since $U$ is closed, it is compact, by~(4).
Since $U$ is open, it is a union of open balls, and so a union of finitely many open balls $B_i$,
by compactness of $U$. By (2), each $B_i$ is $\ModS(\phi_i)$,
for some $\phi_i\in \QFL$.
Then $U=\ModS(\phi)$, where $\phi=\bigvee_i\phi_i$.
\end{proof}

\begin{corollary}
If $L$ is locally finite then all subspaces of the ultrametric space $\SL$ are separable.
\end{corollary}

\begin{proof}
For metric spaces separability is equivalent to existence of a countable base,
which is a hereditary property.
\end{proof}

\section{Gromov--Hausdorff spaces and compactness}

\subsection{Gromov--Hausdorff distance}\label{GH-facts}
First we recall some notions and facts of metric geometry (see \cite[Chapter 7]{BBI}).
We already used above the notion of semi-metric;
we will need a bit more general definition of semi-metric,
in which distances between points can be infinite.

A map $d:X\times X\to \R\cup\{\infty\}$ is called a \emph{semi-metric} on $X$
if $d$ is nonnegative, symmetric, satisfies the triange inequality, and $d(x,x)=0$ for every $x\in X$.
A semi-metric is called a \emph{metric} if
$d(x,y)>0$ for different $x,y\in X$.

A set (or, more generally, a class) equipped with a (semi-)metric
is said to be a (\emph{semi-})\emph{metric space}.
Usually, the set and the space are denoted with the same letter,
and the (semi-)metric of the space $X$ is denoted by $d_X$.

Like a metric, any semi-metric $d$ on $X$ defines a topology on $X$ in a usual way;
this topology is Hausdorff iff $d$ is a metric.

We will use the following easy observations.
Let $X$ and $Y$ be semi-metric spaces, and $f:X\to Y$ be surjective and distance-preserving.
Then
\begin{itemize}
\item
if $A$ is a compact subset of $X$ then $f(A)$ is a compact subset of $Y$, and
\item
if $B$ is a compact subset of $Y$ then $f^{-1}(B)$ is a compact subset of $X$.
\end{itemize}

For a semi-metric $d$ on $X$, the relation $d(x,y)=0$ is an equivalence relation on $X$.
Denote by $[x]$ the equivalence class of $x\in X$, and by $X/d$ the set of all equivalence classes.
Clearly, $([x],[y])\mapsto d(x,y)$ is a well-defined metric on $X/d$; thus
$x\mapsto [x]$ is a surjective and distance-preserving map
from the semi-metric space $X$ to the metric space $X/d$.

The \emph{Hausdorff distance} $d_H(X,Y)$
between subspaces $X$ and $Y$ of a metric space $Z$
is defined to be the infimum of all $r>0$ such that
for any $x\in X$ there is $y\in Y$ with $d(x,y)<r$, and
for any $y\in Y$ there is $x\in X$ with $d(x,y)<r$.
If there is no such $r$ then $d_H(X,Y):=\infty$.
Clearly, $d_H(X,Y)<\infty$ for bounded $X$ and $Y$.

The map $d_H$ is a semi-metric on the set of all subspaces of $Z$.
In general, it is not a metric: for example, $d_H(X,\bar{X})=0$,
for any subspace $X$ of $Z$ and its closure~$\bar{X}$ in $Z$.
However, $d_H$ is a metric on the set of closed subsets of $Z$.

Any two metric spaces $X$ and $Y$ are isometrically embeddable into a third metric space $Z$;
for each such embeddings the Hausdorff distance between the isometric images of $X$ and $Y$ is defined.
The infimum of Hausdorff distances between $X$ and $Y$ for all such embeddings is called
the \emph{Gromov--Hausdorff distance} between $X$ and $Y$ (cf. \cite{BBI, Gro2});
it is denoted by $d_{GH}(X,Y)$. An equivalent, often more convenient,
definition (see \cite{BBI}, Theorem 7.3.25):
\begin{equation}
d_{GH}(X,Y)=\frac{1}{2}\inf
\sup\{|d_X(x_i,x_j)-d_Y(y_i,y_j|:\ i,j\in I\}, \tag{$\star$}
\end{equation}
where the infimum is taken over all maps $i\mapsto x_i$,\ $i\mapsto y_i$
from sets  $I$ onto $X$, $Y$.

The map $d_{GH}$ is a semi-metric on the class of all metric spaces;
we denote the corresponding semi-metric space by $\GH$.

The semi-metric $d_{GH}$ can be extended to a semi-metric
on the class of all semi-metric spaces: for semi-metric spaces $X$ and $Y$ put
$$d_{GH}(X,Y)=d_{GH}(X/d_X, Y/d_Y).$$
It is easy to show that ($\star$) holds for semi-metric spaces $X$ and $Y$ as well.

\subsection{Uniform boundness and compactness}
It is known that any uniformly totally bounded class of compact metric spaces
is precompact in the Gromov-Hausdorff topology~\cite[Theorem 7.4.15]{BBI}.
Here a class of metric spaces $\X$
is called \emph{uniformly totally bounded}~if for every $\varepsilon\ge 0$ there is
a positive integer $n_\varepsilon$ such that
\begin{enumerate}
\item
the diameter of every space in $\X$ is $\le n_0$;
\item
for any $\varepsilon>0$ any space in $\X$
can be covered by $\le n_\varepsilon$ closed balls of radius~$\varepsilon$.
\end{enumerate}

Our goal is to  prove compactness of certain subclasses of $\GH$
using compactness of the Stone space $\SL$ for a certain~$L$.

We call semi-metric spaces satisfying (1) and (2) $\nu$-bounded, where
$$\nu:[0,\infty)\to {\mathbb Z}^{>0},\quad \nu(\varepsilon)=n_\varepsilon.$$
We denote the class of $\nu$-bounded metric spaces by $\X_\nu$.
So a class $\X$ of metric spaces is uniformly totally bounded if $\X\subseteq\X_\nu$, for some $\nu$.

\begin{theorem}\label{nu-bound}
For any $\nu$, the class $\X_\nu$ is compact in $\GH$.
\end{theorem}

We postpone the proof until Subsection~\ref{compactness-nu-bound},
because for that we need a certain correspondence between
semi-metric spaces and structures, which requires some preparatory work.

Theorem~\ref{nu-bound} has a corollary which is a refinement
of the result on precompactness of
any uniformly totally bounded class of compact metric spaces in the Gromov-Hausdorff topology.

\begin{corollary}\label{compact}
For any $\nu$, the class $\Cnu$ of $\nu$-bounded compact metric spaces is compact in $\GH$.
\end{corollary}

\begin{proof}[Proof of Corollary~\ref{compact}]
It suffices to show that the map $X\mapsto\hat{X}$, where $\hat{X}$ is a completion of $X$,
is a surjective distance-preserving map from $\X_\nu$ to $\Cnu$.

First we note that if $X\in\X_\nu$ then
$\hat{X}\in\Cnu$.
If $X\in\X_\nu$ is a dense subspace of a metric space $Y$,
then $Y\in\X_\nu$.
(Indeed, first, $X$ and $Y$ have the same diameter,
and, second, if for some $x_1,\dots,x_n\in X$ and $\varepsilon>0$ the closed balls
$\bar{B}_X(x_i,\varepsilon)$ cover $X$
then the closed balls
$\bar{B}_Y(x_i,\varepsilon)$ cover $Y$
because otherwise the complement of $\bigcup_i\bar{B}_Y(x_i,\varepsilon)$ in $Y$
is open and nonempty but does not meet $X$, contrary to density of $X$ in $Y$.)
So $\hat{X}\in\X_\nu$.
Since a metric space is compact iff it is complete and totally bounded,
$\hat{X}$ is compact. So $\hat{X}\in\Cnu$.

Now we show that $X\mapsto\hat{X}$ maps $\X_\nu$ onto $\Cnu$.
For $Y\in\Cnu$ and $i=1,2,\dots$ choose
$n_i$ closed balls ${\bar B}_Y(y_{ij},1/i)$ which cover~$Y$.
Let $X$ be the subspace of all $y_{ij}$. Then $X\in\X_\nu$ and $Y=\hat{X}$.

Clearly, $d_{GH}(X,\hat{X})=0$; therefore the map $X\mapsto\hat{X}$ preserves $d_{GH}$.
\end{proof}

\subsection{Semi-metric structures}\label{semi-metric-structure}
Now we introduce some relational signature $L_0$,
and associate with any semi-metric space $X$ a set of certain
$L_0$-structures with universe $X$; we call them \emph{$X$-structures}.
An $L_0$-structure, which is an $X$\nobreakdash-structure for some semi-metric space $X$,
will be called a \emph{semi-metric structure}. We call $L_0$ the
\emph{signature of semi-metric structures}.

The signature $L_0$ consists of binary relational symbols $\Rep$, where $\varepsilon\in\R^{>0}$.
An $L_0$-structure $M$ with a universe $X$ is called an $X$\nobreakdash-\emph{structure} if
for any $\varepsilon>0$ and any $x,y\in X$
$$[d_X(x,y)<\varepsilon] \Rightarrow [M\models\Rep(x,y)] \Rightarrow  [d_X(x,y)\le \varepsilon].$$

An example of $X$-structure is the $L_0$-structure $M_X$ on $X$ in which for any $\varepsilon$
$$\Rep^{M_X}=\{(x,y)\in X\times X: d_X(x,y)\le \varepsilon\}.$$

We will use the structure $M_X$ in the proof of Theorem~\ref{nu-bound}.
This example is not unique:
replacing $\le$ with $<$ in the definition of $M_X$, we obtain another example of $X$-structure.

Now we show that any $X$-structure completely determines the space $X$.

For an $L_0$-structure $M$ and $x,y\in M$ denote
\begin{equation*}
d_M(x,y)=
\begin{cases}
\inf\{\varepsilon: M\models\Rep(x,y)\},
&\text{if $M\models\Rep(x,y)$ for some $\varepsilon$;}\\
\infty, &\text{otherwise}.
\end{cases}
\end{equation*}

\begin{proposition}\label{unique}
If $M$ is an $X$-structure then $d_M=d_X$.
In particular, if $M$ is an $X$-structure and $Y$-structure then
the semi-metric spaces $X$ and $Y$ coincide.
\end{proposition}

\begin{proof} Let $x,y\in X$.
For any $\varepsilon>0$ with $M\models\Rep(x,y)$ we have $d_X(x,y)\le\varepsilon$; so
$$d_X(x,y)\le\inf\{\varepsilon: M\models\Rep(x,y)\}=d_M(x,y).$$
Suppose $d_X(x,y)< d_M(x,y)$.
Choose $\varepsilon$ with $$d_X(x,y)<\varepsilon<d_M(x,y).$$
Since $d_X(x,y)<\varepsilon$, we have
$M\models\Rep(x,y)$, and hence $d_M(x,y)\le \varepsilon$, contrary to the choice of $\varepsilon$.
\end{proof}

\begin{proposition}\label{univ-axioms}
The class of semi-metric structures is universally axiomatizable.
\end{proposition}

\begin{proof}
Let $\Gamma$ be the set of universal $L_0$-sentences
\medskip

\begin{enumerate}
\item[(a)]
 $\forall uv\,(R_\delta(u,v)\to \Rep(v,u))$,\quad for $\delta<\varepsilon$;
\medskip

\item[(b)]
 $\forall uvw\,((\Rep(u,v)\wedge R_\delta(v,w))\to R_\eta(u,w))$,
for $\varepsilon+\delta<\eta$;
\medskip

\item[(c)]
$\forall u\,\Rep(u,u)$,
\medskip

\item[(d)]
$\forall uv\,(R_\delta(u,v)\to \Rep(u,v))$,\quad for $\delta<\varepsilon$;
\end{enumerate}
\smallskip

\noindent
where $\delta,\varepsilon ,\eta$ run over $\mathbb{R}^{>0}$.
\medskip

It is easy to check that any $X$-structure is a model of $\Gamma$.
We show that any model $M$ of $\Gamma$ is an $X$-structure for some semi-metric space $X$.

Let $M$ be a model of $\Gamma$. We show that
$d_M$ is a semi-metric on the universe of~$M$.
Obviously, $d_M$ is non-negative.
The axiom (a) implies that $d_M$ is symmetric. Indeed, suppose, say, $d_M(x,y)<d_M(y,x)$. Choose $\varepsilon$
with $$d_M(x,y)<\varepsilon<d_M(y,x).$$
By definition of $d_M$, there is $\delta<\varepsilon$ with $M\models R_\delta(x,y)$.
By (a), $M\models \Rep(y,x)$. Hence $d_M(y,x)\le\varepsilon$. Contradiction.

The axiom (b) implies  that $d_M$ satisfies the triangle inequality.
Towards a contradiction, suppose $$d(x,z)>d_M(x,y)+d_M(y,z).$$
Choose reals $\alpha$ and $\beta$
such that $$d_M(x,y)<\alpha,\quad d_M(y,z)<\beta,\quad \alpha+\beta<d_M(x,z).$$
By definition of $d_M$, there are $\varepsilon$ and $\delta$ such that $M\models\Rep(x,y)$ and
$M\models R_\delta(y,z)$.
By (b), $M\models R_{\alpha+\beta}(x,y)$.
Then $d_M(x,z)\le \alpha+\beta$. Contradiction.

By~(c), $d_M(x,x)=0$ for any $x\in M$.
Let $X$ be the semi-metric space which is the
universe of $M$ equipped with $d_M$. We show that $M$ is an $X$\nobreakdash-structure.
By definition of $d_M$, if $M\models\Rep(x,y)$ then $d_M(x,y)\le\varepsilon$;
if $d_M(x,y)<\varepsilon$ then $M\models R_\delta(x,y)$ for some $\delta<\varepsilon$, and therefore
$M\models\Rep(x,y)$, by (d).
\end{proof}

\subsection{Compactness of $\X_\nu$}\label{compactness-nu-bound}

In this subsection we give a proof of Theorem~\ref{nu-bound}.

Let $\Ynu$ be the class of $\nu$-bounded semi-metric spaces of cardinality at most $2^{\aleph_0}$.
Then $X\mapsto X/d_X$ is a distance-preserving map from $\Ynu$ to $\Xnu$.
Since any $\nu$-bounded metric space is of cardinality at most $2^{\aleph_0}$, the map is surjective.
Therefore $\Xnu$ is compact iff $\Ynu$ is compact.
We will prove compactness of $\Ynu$.

For that we define an extention $L$ of the signature $L_0$
of semi-metric structures by some constants, a universally axiomatizable subclass of $\ML$,
and a continuous surjective map  from that subclass onto $\Ynu$.
Since the subclass is compact in the Stone topology on $\ML$,
due to results of Section~\ref{stone}, this implies compactness of $\Ynu$.

Let $L=L_0(C)$, where  $C$ is the union of a family of pairwise disjoint sets of constant symbols
$\{C_\varepsilon : \varepsilon\ge 0\}$, with
$|C_0|=2^{\aleph_0}$, and $|C_\varepsilon| = n_\varepsilon$ for $\varepsilon>0$.

Let $\Gamma_\nu$ be the union of $\Gamma$ and the set of universal $L$-sentences
\medskip

\begin{enumerate}

\item[(e)]
$\forall uv\,R_{n_0}(u,v)$;
\medskip

\item[(f)]
$\forall u\,\bigvee\{\Rep(u,c): c\in C_\varepsilon\},\quad \varepsilon>0.$
\end{enumerate}
\medskip

Denote by $\Mnu$ the class $\ModM(\Gamma_\nu)$.
By results of Section~\ref{stone}, the class $\Mnu$ is compact in Stone topology on $\ML$.

\begin{lemma}\label{correspondence}
$(1)$ For any $X\in\Ynu$ there exists $M\in\Mnu$ such that
the $L_0$-reduct of $M$ is an $X$-structure;

$(2)$ for any $M\in\Mnu$ there is
a unique $X\in\Ynu$ such that the $L_0$-reduct of $M$ is an $X$-structure.
\end{lemma}

\begin{proof}
(1) Let $X\in\Ynu$. Since $|X|\le 2^{\aleph_0}$ and $X$ is $\nu$-bounded,
there is a $f:C\to X$ such that $f(C_0)=X$, and for every $\varepsilon>0$
the closed balls of radius $\varepsilon$ centered at $f(c)$, where $c\in C_\varepsilon$, cover $X$.

The $X$-structure $M_X$ defined in Subsection~\ref{semi-metric-structure} has the following property:
$$M_X\models\Rep(x,y)\quad\text{iff}\quad d_X(x,y)\le\varepsilon,$$
for all $x,y\in X$ and  all $\varepsilon<0$. By Proposition~\ref{univ-axioms}, $M_X$ is a model of $\Gamma$.

Consider the $L$-expansion $M$ of $M_X$ such that
$c^M=f(c)$ for all $c\in C$.
Then $M$  satisfies (e) and (f), by the choice of $f$.
Since $f(C_0)=X$, the $L$-structure $M$ is minimal.
Thus $M\in\Mnu$, and its $L_0$-reduct is the $X$-structure $M_X$.
\medskip

(2) As $M$ satisfies $\Gamma$, the $L_0$-reduct of $M$ is an $X$-structure for some
semi-metric space $X$, which is unique, by Proposition~\ref{unique}.
Since $M$ is a minimal $L$\nobreakdash-structure,
$|M|\le 2^{\aleph_0}$, and so $|X|\le 2^{\aleph_0}$.
As $M$ satisfies (e), the diameter of $X$ is $\le n_0$.
Since $M$ satisfies (f), $X$ is covered by the close balls of radius $\varepsilon$ centered at
$c^M$ with $c\in C_{\varepsilon}$.
Thus $X\in\Ynu$.
\end{proof}

For $M\in\Mnu$ let $\chi(M)$ be the unique $X\in\Ynu$ such that
the $L_0$\nobreakdash-reduct of $M$ is an $X$-structure, which exists by Lemma~\ref{correspondence}(2).
The map
$$\chi: \Mnu\to \Ynu$$
is surjective, by Lemma~\ref{correspondence}(1).
Now, to complete the proof of compactness of $\Ynu$, it suffices to prove

\begin{lemma}
The map $\chi$ is continuous.
\end{lemma}

\begin{proof}
To prove that $\chi$ is continuous at $M_0\in\Mnu$, we need to show that for any $\alpha>0$ there is
$\psi\in\QFL$ with $M_0\models\psi$ such that for any $N\in\Mnu$ with $N\models\psi$
$$d_{GH}(\chi(M_0),\chi(N))<\alpha.$$
For any $\alpha>0$ we construct a finite $\Phi\subseteq\QFL$ such that, for any $M,N\in\Mnu$,
$$(M\models\phi\ \text{iff}\  N\models\phi\ \ \text{for all $\phi\in\Phi$})
\ \Rightarrow\ d_{GH}(\chi(M),\chi(N))<\alpha;$$
then we can take as $\psi$ the conjunction of all sentences from $\Phi\cup\{\neg\phi:\phi\in\Phi\}$ that hold in $M_0$.

Choose $\varepsilon$ with $0<\varepsilon< n_0$ and $5\varepsilon/2<\alpha$.
Let $m$ be the integer with $$0<m\varepsilon<n_0\le (m+1)\varepsilon.$$
Let $\Phi$ be the set of all sentences  $R_{i\varepsilon}(a,b)$, where
$i\in\{1,\dots,m\}$ and ${a,b\in C_\varepsilon}$.
We show that the finite set $\Phi$ satisfies the required conditions.

Let $M,N\in\Mnu$. Denote $\chi(M)=X$ and $\chi(N)=Y$;
so the $L_0$-reduct of $M$ is an $X$-structure, and
the $L_0$-reduct of $N$ is a $Y$-structure.
Let
$$
X_\varepsilon=\{c^M: c\in C_\varepsilon\},\qquad
Y_\varepsilon=\{c^N: c\in C_\varepsilon\}.
$$
Since $X$ is covered by the closed balls centered at points of the set $X_\varepsilon$, we have
$d_{GH}(X,X_\varepsilon)\le\varepsilon$.
Similarly,
$d_{GH}(Y,Y_\varepsilon)\le\varepsilon$.
Therefore by the triangle inequality,
$$
d_{GH}(X,Y)\le \varepsilon+ d_{GH}(X_\varepsilon, Y_\varepsilon)+\varepsilon.
$$
Hence it suffices to show that if $M\models \phi$  iff  $N\models \phi$  for all $\phi\in\Phi$,
then $$d_{GH}(X_\varepsilon, Y_\varepsilon)\le\varepsilon/2,$$ because this implies
$$
d_{GH}(X,Y)\le \varepsilon+ \varepsilon/2+\varepsilon=5\varepsilon/2<\alpha.
$$
To prove $d_{GH}(X_\varepsilon, Y_\varepsilon)\le\varepsilon/2,$
it suffices to show
that for any $a,b\in C_\varepsilon$
$$|d_X(a^N,b^N)-d_Y(a^M,b^M)|\le\varepsilon,$$
by ($\star$) from Subsection~\ref{GH-facts}.
The latter inequality holds because, first,
the diameters of $M$ and $N$ are $\le n_0$, and so
$$0\le d_X(a^N,b^N),\ d_Y(a^M,b^M)\le n_0,$$
and, second,
none of the numbers $\varepsilon ,2\varepsilon,\dots, m\varepsilon$
can be strictly between
$d_X(a^N,b^N)$ and $d_Y(a^M,b^M)$: if, say,
$$d_X(a^N,b^N)<i\varepsilon< d_Y(a^M,b^M),$$
then $N\models R_{i\varepsilon}(a,b)$ and $M\nvDash R_{i\varepsilon}(a,b)$.
The lemma is proven.
\end{proof}

The proof of Theorem~\ref{nu-bound} is completed.\hfill $\square$

\end{document}